\documentclass[10pt]{article}
\usepackage{amssymb, amsfonts, amsthm, mathrsfs}
\usepackage[super]{nth}
\usepackage[pagebackref,colorlinks=true]{hyperref}
\usepackage{graphicx, xcolor}
\usepackage[all]{xy}
\usepackage[leqno]{amsmath}
\usepackage{titlesec}

\theoremstyle{definition}
\newtheorem{thm}{Theorem}

\newtheorem{rmk}[thm]{Remark}

\titleformat{\section}{\Large\bfseries}{\thesection}{.5em}{}

\makeatletter

\def\R{\mathbb{R}}

\def\p{\partial}

\def\ud{\mathrm{d}}

\makeatother

\title{\textbf{An Estimate of the First Eigenvalue of a Schr\"odinger Operator on Closed Surfaces}}
\author{Teng Fei and Zhijie Huang}
\date{}

\begin{document}
\maketitle{}

Let $\Sigma$ be a closed surface equipped with an arbitrary Riemannian metric $g$ and let $\Delta$ be the associated Laplace-Beltrami operator on $\Sigma$. In this short note, we establish an estimate for the first eigenvalue of the Schr\"odinger operator $-\Delta+2\kappa$ on $\Sigma$, where $\kappa$ is the Gauss curvature of the given metric.  This operator appears as the stability operator (Jacobi operator) of minimal surfaces in $\R^3$ or the flat $3$-torus $T^3$. Our method is based on the work of Schoen-Yau \cite{schoen1983}.
\begin{thm}~\\
Let $\lambda_1$ be the first eigenvalue of the operator $-\Delta+2\kappa$ on $\Sigma$ and let $D$ be the diameter of $\Sigma$. For any parameter $0<\mu<2$, we have the following estimate:
\begin{equation}\label{esti}
\lambda_1\leq\max\left(\frac{2\mu-1}{\mu}\kappa\right)+\frac{4-\mu}{\mu(4-2\mu)}\frac{\pi^2}{D^2}.
\end{equation}
In particular, by setting $\mu=1/2$, we get an upper bound depending only on $D$:
\begin{equation}
\label{maines}\lambda_1\leq\frac{7\pi^2}{3D^2}.
\end{equation}
\end{thm}

\begin{proof}
Let $q>0$ be the first eigenfunction of the operator $-\Delta+2\kappa$, so we have
\[-\Delta q+2\kappa q=\lambda_1q.\]

Fix two points on $\Sigma$, for any curve $\gamma$ joining them, consider the functional $\int_\gamma v$, where $v$ is a fixed positive function on $\Sigma$. Let $\gamma$ be a minimizer of this functional, which always exists because we can view it as a geodesic connecting the two given points under a conformally changed metric. Denote by $s$ the arc length parameter of $\gamma$. Let $\eta=\varphi\cdot n$ be a normal variational vector field along $\gamma$, where $n$ is a fixed unit normal vector field of $\gamma$ and $\varphi$ is a smooth function on $\gamma$ vanishing at two end points. Denote by $\tau$ the geodesic curvature of $\gamma$, we have
\[\nabla_{\frac{\p}{\p s}}n=\tau\frac{\p}{\p s}.\]
From the vanishing of the first variation, we get
\[v_n+v\tau=0,\]
where $v_n$ is the normal derivative of $v$ along $\gamma$. Furthermore, the second variation gives us
\[\int_\gamma\left(-v\varphi\varphi''+\varphi^2\left(\frac{1}{2}v''-v\kappa-2v\tau^2+\nabla^2v(n,n)\right)\right)\ud s\geq0\]
for any test function $\varphi$. Notice that
\[\nabla^2v\left(\frac{\p}{\p s},\frac{\p}{\p s}\right)=v''+\tau v_n=v''-\tau^2v,\]
therefore we can rewrite the above inequality as
\[\int_\gamma\left(-v\varphi\varphi''+\varphi^2\left(\Delta v-\frac{1}{2}v''-\tau^2v-\kappa v\right)\right)\ud s\geq0.\]
Let $L$ be the operator given by
\[L\varphi=-v\varphi''+\varphi(\Delta v-\frac{1}{2}v''-\tau^2v-\kappa v),\]
then $L$ is nonnegative. Let $h>0$ be the first eigenfunction of $L$, hence we have
\[-vh''+h(\Delta v-\frac{1}{2}v''-\tau^2v-\kappa v)\geq 0,\]
or equivalently
\[vh''+\frac{1}{2}hv''\leq h(\Delta v-\tau^2v-\kappa v).\]
Now take $v=q^\mu$ for some $0<\mu<2$. We have
\[\Delta v=\Delta(q^\mu)=\mu q^{\mu-1}\Delta q+\mu(\mu-1)q^{\mu-2}|\nabla q|^2=\mu(2\kappa-\lambda_1)v+\mu(\mu-1)q^{\mu-2}|\nabla q|^2,\]
therefore
\[vh''+\frac{1}{2}hv''\leq hv(\kappa(2\mu-1)-\mu\lambda_1-\tau^2)+\mu(\mu-1)q^{\mu-2}|\nabla q|^2h.\]
Notice that
\[|\nabla q|^2=(q')^2+(q_n)^2\]
and we have
\[q_n=\frac{-\tau}{\mu}q,\quad\quad q'=\frac{v'q}{\mu v},\]
hence
\[vh''+\frac{1}{2}hv''\leq hv\left(\kappa(2\mu-1)-\mu\lambda_1-\frac{1}{\mu}\tau^2\right)+\frac{\mu-1}{\mu}\frac{(v')^2}{v}h.\]
Since $\mu>0$, we have
\[h^{-1}h''+\frac{1}{2}v^{-1}v''\leq \kappa(2\mu-1)-\mu\lambda_1+\frac{\mu-1}{\mu}\big((\log v)'\big)^2.\]
By direct calculation, we know
\[h^{-1}h''+\frac{1}{2}v^{-1}v''=(\log h+\frac{1}{2}\log v)''+((\log h)')^2+\frac{1}{2}((\log v)')^2,\]
so we get
\[\big((\log h)'\big)^2+\left(\frac{1}{\mu}-\frac{1}{2}\right)\big((\log v)'\big)^2+\mu\lambda_1-(2\mu-1)\kappa\leq -\Big(\log h+\frac{1}{2}\log v\Big)''.\]
Let $\psi$ be any test function on $\gamma$. Multiplying the above inequality by $\psi^2$ and integration by part, we get
\[\begin{split}&\int_\gamma\psi^2\left[((\log h)')^2+\left(\frac{1}{\mu}-\frac{1}{2}\right)((\log v)')^2\right]\ud s+\int_\gamma\psi^2(\mu\lambda_1-(2\mu-1)\kappa)\ud s\\
\leq& \int_\gamma 2\psi\psi'\Big((\log h)'+\frac{1}{2}(\log v)'\Big)\ud s\\ \leq &A\int_\gamma\psi^2\Big((\log h)'+\frac{1}{2}(\log v)'\Big)^2\ud s+A^{-1}\int_\gamma(\psi')^2\ud s.\end{split}\]
As $\mu<2$, we may choose a suitable $A$ such that
\[A\Big((\log h)'+\frac{1}{2}(\log v)'\Big)^2\leq ((\log h)')^2+\left(\frac{1}{\mu}-\frac{1}{2}\right)((\log v)')^2.\]
The best constant $A$ is given by $\dfrac{4-2\mu}{4-\mu}$.
So we get
\[(\mu\lambda_1-\max\left((2\mu-1)\kappa\right))\int_\gamma \psi^2\ud s\leq \frac{4-\mu}{4-2\mu}\int_\gamma(\psi')^2\ud s\]
for any test function $\psi$.

Suppose $\gamma$ has length $l$, then the first eigenvalue of $-\dfrac{\ud^2}{\ud s^2}$ on $\gamma$ is $\dfrac{\pi^2}{l^2}$, so we get
\[\lambda_1\leq\max\left(\frac{2\mu-1}{\mu}\kappa\right)+\frac{4-\mu}{\mu(4-2\mu)}\frac{\pi^2}{l^2}.\]
Because we have the freedom to choose the two endpoints of $\gamma$, we get the desired estimate.
\end{proof}
\begin{rmk}~\\
As far as the authors can find in the literature, upper bounds of $\lambda_1$ (e.g. \cite{grigoryan2004}) typically take the form
\[\lambda_1\leq \frac{4\pi\chi}{\textrm{Area}},\] where $\chi$ is the Euler characteristics of $\Sigma$. Basically one can feed in the constant function to the Rayleigh quotient to derive this estimate.

If $\Sigma$ is topologically a sphere, our estimate (\ref{maines}) is in some sense better than the ones in the literature. Because the isoperimetric inequality (e.g. \cite{shioya2015}) bounds the area from above by a constant times diameter square but not the other way around.

If $\Sigma$ is a negatively-curved closed surface, it seems possible to produce a negative upper bound for $\lambda_1$ by putting $\mu>1/2$ in (\ref{esti}). However we are unable to compare our result with the ones in literature.

Our method certainly applies to other more general Schr\"odinger operators on closed surfaces. We used $-\Delta+2\kappa$ as our example because it is intrinsically defined and is related to the stability of minimal surfaces as we stated at the beginning of this note. We are also motivated by our study of Strominger system \cite{fei2017} where spectral properties of this operator play an central role.
\end{rmk}

\begin{rmk}~\\
	Let $\Sigma \hookrightarrow S^3(1)$ be a surface of constant mean curvature $H$ with second fundamental form $A$ and principal curvatures $\kappa_1, \kappa_2$, then by Gauss equation, we have $2 \kappa= 2 +4H^2 - |A|^2 = 2  +2 \kappa_1 \kappa_2 $.  The associated Jacobi operator is
	$$ J  = -\Delta  - (|A|^2 + 2) = -\Delta + 2\kappa - 4 - 4H^2 $$
	The first eigenvalue of $J$ is related to $\lambda_1$ by $\lambda_1^J = \lambda_1  - 4 - 4H^2$.
	Simons \cite{simons1968} proved that if $H=0$, then $\lambda_1^J \leq -n$. Al\'ias, Barros and Brasil \cite{alias2005} proved that either $\lambda_1^J = -n(1+H^2)$ and $\Sigma$ is totally umbilical or
$$\lambda_1^J \leq -2n(1+H^2)  + \frac{n(n-2)|H|}{\sqrt{n(n-1) }}\max \sqrt{|A|^2 - nH^2}.$$
In particular, when $n=2$, their result reduces to either $\lambda_1^J = -2(1+H^2)$ when $\Sigma$ is totally umbilical or $\lambda_1^J \leq -4(1+H^2)$. Hence, their result implies that either $\lambda_1 =2 + 2H^2$ if $\Sigma$ is umbilical or $\lambda_1 \leq 0$ if $\Sigma$ is not umbilical. For the umbilical case, $\Sigma$ has to be a sphere with radius $1/(1+H^2)$. In this case, $\kappa = 1 + H^2$ and $\lambda_1 = 2\kappa = 2+ 2H^2.$ For the umbilical case, the genus of $\Sigma$ is greater or equal to 1, then we have that the first eigenvalue of $-\Delta + 2\kappa $ is non-positive. It is an easy consequence of Gauss-Bonnet theorem if we take the constant function as test function. Our result here is somehow more general than theirs. The estimate here is intrinsic, which does not depend on the ambient space.
\end{rmk}

\noindent\textbf{Acknowledgements} The authors would like to thank Prof. D.H. Phong and Prof. S.-T. Yau for their constant encouragement and help. The authors are also indebted to inspiring discussions with S. Picard and M.-T. Wang.

\bibliographystyle{plain}

\bibliography{C:/Users/Piojo/Dropbox/Documents/Source}

\bigskip

\noindent Department of Mathematics, Columbia University, New York, NY 10027, USA\\
tfei@math.columbia.edu, zjhuang@math.columbia.edu,

\end{document}